\documentclass[a4paper,12pt]{article}
    \usepackage[top=2.5cm,bottom=2.5cm,left=2.5cm,right=2.5cm]{geometry}
    \usepackage{cite, amsmath, amssymb}
    \usepackage[margin=1cm,%
                font=small,%
                format=hang,%
                labelsep=period,%
                labelfont=bf]{caption}
\usepackage{enumitem,url}
\usepackage{graphicx}

\newtheorem{theorem}{Theorem}
\newtheorem{theo}{Theorem}
\newtheorem{lemma}[theo]{Lemma}

\newtheorem{remark}[theo]{Remark}

\newenvironment{proof}{\par \noindent \textbf{Proof: }}{\QED \par \bigskip \par}
\newcommand{\QED}{\hfill$\square$}

\allowdisplaybreaks[4]

\begin{document}

\baselineskip=0.30in

\vspace*{10mm}

\begin{center}
{\Large \bf \boldmath Minimum Atom-Bond Sum-Connectivity\\[2mm]
Index of Trees With a Fixed Order and/or\\[2mm]
Number of Pendent Vertices}

\vspace{10mm}

{\large \bf Tariq A. Alraqad$^{1}$, Igor \v{Z}. Milovanovi\'c$^{2}$, Hicham Saber$^{1}$,\\ Akbar Ali$^{1}$, Jaya P. Mazorodze$^{3}$, Adel A. Attiya$^{1}$}

\vspace{6mm}

\baselineskip=0.23in

$^1${\it Department of Mathematics, Faculty of Science,\\ University of Ha\!'il, Ha\!'il, Saudi Arabia}\\
{\tt t.alraqad@uoh.edu.sa, hicham.saber7@gmail.com, akbarali.maths@gmail.com, a.attiya@uoh.edu.sa}\\[3mm]
$^2${\it Faculty of Electronic Engineering,\\ University of Ni\v{s}, Ni\v{s}, Serbia}\\
{\tt igor.milovanovic@elfak.ni.ac.rs}\\[3mm]
$^3${\it Department of Mathematics,\\ University of Zimbabwe, Harare, Zimbabwe}\\
{\tt mazorodzejaya@gmail.com }

\makeatletter

\def\@makefnmark{}

\makeatother


\vspace{4mm}

\baselineskip=0.23in

{\bf Abstract }
\end{center}
\noindent
Let $d_u$ be the degree of a vertex $u$ of a graph $G$. The atom-bond sum-connectivity (ABS) index of a graph $G$ is the sum of the numbers $(1-2(d_v+d_w)^{-1})^{1/2}$ over all edges $vw$ of $G$. This paper gives the characterization of the graph possessing the minimum ABS index in the class of all trees of a fixed number of pendent vertices;
the star is the unique extremal graph in the mentioned class of graphs. The problem of determining graphs possessing the minimum ABS index in the class of all trees with $n$ vertices and $p$ pendent vertices is also addressed; such extremal trees have the maximum degree $3$ when $n\ge 3p-2\ge7$, and the balanced double star is the unique such extremal tree for the case $p=n-2$.\\[3mm]

\baselineskip=0.35in

\section{Introduction}

A property of a graph that is preserved by isomorphism is known as a graph invariant \cite{Gross-05}. The order and degree sequence of a graph are examples of graph invariants. The graph invariants that assume only numerical values are usually referred to as topological indices in chemical graph theory \cite{Wagner-18}.

For evaluating the extent of branching of the carbon-atom skeleton of saturated hydrocarbons, Randi\'c \cite{Randic-75} devised a topological index and called it as the branching index, which nowadays is known as the connectivity index (also, the Randi\'c index).


 The connectivity index of a graph $G$ is the following number:
\[
\sum_{vw\in E(G)} \frac{1}{\sqrt{d_vd_w}}\,,
\]
where $d_v$ and $d_w$ denote the degrees of the vertices $v$ and $w$ of $G$  respectively, and $E(G)$ denotes the set of edges of a graph $G$. It is believed that the connectivity index is the most-studied topological index (in both theoretical and applied aspects) \cite{Gutman-13}. Detail about the study of the connectivity index can be found in the survey papers  \cite{Li-MATCH-08,Randic-01}, books \cite{new1,new3}, and related papers cited therein.

Because of the success of the connectivity index, many modified versions of this index have been introduced in the literature. The atom-bond connectivity (ABC) index \cite{Estrada-98,Estrada-08} and the sum-connectivity (SC) index \cite{Zhou-JMC-09} are among the well-studied modified versions of the connectivity index. The ABC and SC indices of a graph are defined as
\[
ABC(G)=\sum_{vw\in E(G)} \sqrt{\frac{d_v+d_w-2}{d_vd_w}}\,,
\]
and
\[
SC(G)=\sum_{vw\in E(G)} \frac{1}{\sqrt{d_v+d_w}}.
\]
The readers interested in detail about the ABC and SC indices are referred to the survey papers \cite{Ali-DML-21} and \cite{Ali-MATCH-19}, respectively.

Using the main idea of the SC index, a modified version of the ABC index was proposed in \cite{ABS-JMC} recently and it was referred to as the atom-bond sum-connectivity (ABS) index. The ABS index of a graph $G$ is defined as
$$
ABS(G)=\sum_{uv\in E(G)} \sqrt{\frac{d_u+d_v-2}{d_u+d_v}}
= \sum_{uv\in E(G)}\sqrt{1-\frac{2}{d_u+d_v}}\,.
$$
Although the ABS index is a special case of a general topological index considered in \cite{Tang-DAM-16}, no result reported in \cite{Tang-DAM-16} covers the ABS index.
The graphs possessing the maximum/minimum ABS index in the class of all (i) (molecular) trees (ii) general graphs, with a given order, were characterized in \cite{ABS-JMC}. Analogous results for unicyclic graphs were reported in \cite{ABS-EJM}, where the chemical applicability of the ABS index was also investigated.

A vertex of degree one in a tree $T$ is called a pendent vertex and a vertex of degree at least three in $T$ is called a branching vertex. A path $P$ in a tree $T$ connecting a branching vertex and a pendent vertex is called a pendent path provided that every other vertex (if exists) of $P$ has degree two in $T$. A path $P$ in a tree $T$ is said to be an internal path if it connects two branching vertices and every other vertex (if exists) of $P$ has degree two in $T$. A tree with one non-pendent vertex is called a star. A double star is a tree with exactly two non-pendent vertices. A double star tree with non-pendent vertices $u$ and $v$ is called balanced if $|d_u-d_v|\leq 1$. For a general reference on graph theory see \cite{Bondy08}.

Ali et. al. \cite{ABS-EJM} posed a problem asking to determine trees possessing the minimum value of the ABS index among all trees with a fixed number of pendent vertices. The main goal of the present paper is to determine trees possessing the minimum value of the $ABS$ index in two classes of trees.  For positive integers $n$ and $p$, $\Gamma_p$ denotes the class of all trees with $p$ pendent vertices and $\Gamma_{n,p}$ denotes the class of all trees of order $n$ and $p$ pendent vertices. In section \ref{sec1} we give a complete solution to the problem posed by Ali et. al. \cite{ABS-EJM}, where we show that the star graph $S_{p+1}$ uniquely attains the minimum value of the $ABS$ index in the class $\Gamma_p$. In section \ref{sec2} we provide results on trees that minimize the value of the $ABS$ index in $\Gamma_{n,p}$.

\section{Trees with a Fixed Number of Pendent Vertices}\label{sec1}

We will need the next  already known result.

\begin{lemma}\label{lem-AZI+e} \cite[Corollary 8]{ABS-JMC}
Let $u$ and $v$ be non-adjacent vertices in a connected graph $G$. Then $ABS(G + uv) > ABS(G)$, where $G+uv$ is the graph obtained from $G$ by adding the edge $uv$.
\end{lemma}

\begin{lemma}\label{lem-1}
Let $p\geq2$ be an integer. If $T^{\ast}$ is a tree attaining the minimum value of the $ABS$ index in the class $\Gamma_p$, then $T^{\ast}$ has no vertex of degree $2$.
\end{lemma}

\begin{figure}[!ht]
 \centering
  \includegraphics[width=0.45\textwidth]{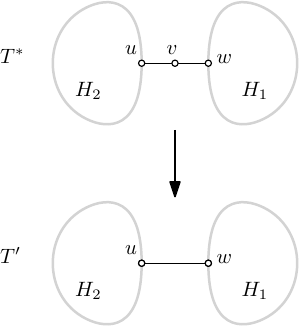}
   \caption{The graph transformation used in the proof of Lemma \ref{lem-1}. The subtree $H_1$ may or may not consist of only one vertex $w$.}
    \label{Fig-1}
     \end{figure}

\begin{proof}
Suppose to the contrary that $T^{\ast}$ has at least one vertex of degree $2$. Take $v\in V(T^{\ast})$ such that $N(v)=\{u,w\}$ and $d_u\ge d_w\ge1$. Let $T'$ be the tree formed by removing the vertex $v$ (and its incident edges) and adding the edge $uw$ (see Figure \ref{Fig-1}). In what follows, by $d_x$ we denote the degree of a vertex $x$ in $T^{\ast}$. Using the definition of the ABS index, we have
\[
 ABS(T^{\ast})-ABS(T')=\sqrt{1-\frac{2}{d_u+2}}+\sqrt{1-\frac{2}{d_w+2}}-\sqrt{1-\frac{2}{d_u+d_{w}}} ~ >0,
\]
a contradiction to the assumption that $T^{\ast}$ attains the minimum value of the  ABS index among all trees with $p$ pendent vertices.

\end{proof}

\begin{lemma}\label{lem-2}
The function $f$ defined by
\begin{align*}
f(x,y)&=(x-1)\left( \sqrt{1-\frac{2}{x+1}}-\sqrt{1-\frac{2}{x+y-1}\,}\, \right)\nonumber\\[2mm]
 &\quad\ +(y-1)\left( \sqrt{1-\frac{2}{y+1}}-\sqrt{1-\frac{2}{x+y-1}\,}\, \right)\nonumber\\[2mm]
 &\quad\ +\sqrt{1-\frac{2}{x+y}}
\end{align*}
with $x\ge y\ge3$, is strictly decreasing on $y$.
\end{lemma}

\begin{proof}
We have
\begin{align}\label{Eq-lem-2-1}
\frac{\partial f}{\partial y} &= \frac{1}{(x+y)^{3/2}\sqrt{x+y-2}} - \frac{x+y-2}{(x+y-1)^{3/2}\sqrt{x+y-3}}  \nonumber\\[3mm]
&\quad \ -\sqrt{\frac{x+y-3}{x+y-1}} + g(y),
\end{align}
where
\[
g(y)=\frac{y+2}{y+1}\sqrt{\frac{y-1}{y+1}}\,.
\]
Certainly, the function $g$ is strictly increasing on $y$ because $y\ge3$. Hence, $g(y) < g(x+y-2)$ as $x\ge3$. Consequently, Equation \eqref{Eq-lem-2-1} gives
\begin{align}\label{Eq-lem-2-2}
\frac{\partial f}{\partial y} &< \frac{1}{(x+y)^{3/2}\sqrt{x+y-2}} - \frac{x+y-2}{(x+y-1)^{3/2}\sqrt{x+y-3}}  \nonumber\\[3mm]
&\quad \ -\sqrt{\frac{x+y-3}{x+y-1}} + \frac{x+y}{x+y-1}\sqrt{\frac{x+y-3}{x+y-1}}\nonumber\\[3mm]
&= \frac{1}{(x+y)^{3/2}\sqrt{x+y-2}} - \frac{1}{(x+y-1)^{3/2}\sqrt{x+y-3}}\,.
\end{align}
Since the function $\psi$ defined by
\[
\psi(t) = \frac{1}{t^{3/2}\sqrt{t-2}}
\]
is strictly decreasing for $t>2$, from \eqref{Eq-lem-2-2} it follows that $\frac{\partial f}{\partial y}<0$.

\end{proof}

\begin{lemma}\label{p1}
Let $$f(x,y)=\sqrt{1-\dfrac{2}{x+y}}\,,$$
and define the function $\psi(x,y;s)$ as $$\psi(x,y;s)=f(x+s,y)-f(x,y),$$ where $x,y\geq 1$ and $s>0$. Then $\psi(x,y;s)$ is strictly decreasing on $x$ and on $y$.
\end{lemma}
\begin{proof}
Since $$\psi(x,y;s)=f(x+s,y)-f(x,y)=f(x,y+s)-f(x,y),$$ it suffices to show the case of $x$. Note that the first partial derivatives of $f$ are calculated as
$$\frac{\partial f}{\partial x}(x,y)=\frac{\partial f}{\partial y}(x,y)=(x+y-2)^{-\frac{1}{2}}(x+y)^{-\frac{3}{2}},$$
which are both strictly decreasing in $x$. This implies that
\[\frac{\partial }{\partial x}\psi(x,y;s)=\frac{\partial f}{\partial x}(x+s,y)-\frac{\partial f}{\partial x}(x,y)<0\] and
\[\frac{\partial }{\partial y}\psi(x,y;s)=\frac{\partial f}{\partial y}(x+s,y)-\frac{\partial f}{\partial y}(x,y)<0.\]
Thus $\psi(x,y;s)$ is strictly decreasing on $x$ and on $y$.
\end{proof}

\begin{theorem}\label{thm-1}
Let $p\geq 2$ be an integer. Then for every $T\in \Gamma_p$, $$ABS(T)\geq p\sqrt{\frac{p-1}{p+1}},$$ with equality holds if and only if $T\cong S_{p+1}$.
\end{theorem}

\begin{figure}[!ht]
 \centering
  \includegraphics[width=0.5\textwidth]{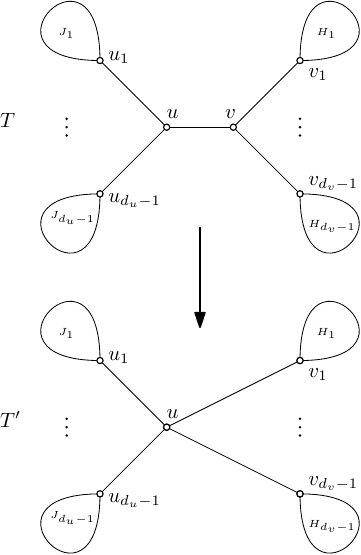}
   \caption{The graph transformation used in the proof of Theorem \ref{thm-1}. For $i\in \{1,\ldots,d_u-1\}$ and $j\in \{1,\ldots,d_v-1\}$, the subtree $J_i$ may or may not consist of only one vertex $u_i$ and the subtree $H_j$ may or may not consist of only one vertex $v_j$, respectively.}
    \label{Fig-2}
     \end{figure}

\begin{proof}
Let $T$ be a graph attaining the minimum $ABS$ value among all trees with $p$ pendent vertices. By Lemma \ref{lem-1}, $T$ has no vertex of degree $2$. We claim that $T$ has only one vertex of degree greater than $2$. Contrarily, we assume that $T$ contains at least two vertices of degree greater than $2$. Among the vertices of degrees at least $3$, we pick $u,v\in V(T)$ such that $uv\in E(G)$ (Lemma \ref{lem-1} guarantees the existence of the vertices $u$ and $v$ when $T$ has at least one pair of vertices of degrees greater than $2$). Without loss of generality, we suppose that $d_u\ge d_v$. Let $v_1,\cdots,v_{d_v-1}$ be the neighbors of $v$ different from $u$. Construct a new tree $T'$ by dropping the vertex $v$ (and its incident edges) and inserting the edges $v_1u,\cdots,v_{d_v-1}u$. Certainly, both the trees $T$ and $T'$ have the same number of pendent vertices. However, in the following, we show that $ABS(T)>ABS(T')$, which gives a contradiction to the minimality of $ABS(T)$ and hence $T$ must contain exactly one vertex of degree greater than $2$, as desired.

In what follows, by $d_w$ we denote the degree of a vertex $w$ in $T$. If $u_1,\cdots,u_{d_u-1}$ are the neighbors of $u$ different from $v$, then
\newpage
\begin{align}\label{Eq-1}
 ABS(T)-ABS(T')&=\sum_{i=1}^{d_u-1}\left( \sqrt{1-\frac{2}{d_u+d_{u_i}}}-\sqrt{1-\frac{2}{d_u+d_v+d_{u_i}-2}\,}\, \right)\nonumber\\[2mm]
 &\quad\ +\sum_{j=1}^{d_v-1}\left( \sqrt{1-\frac{2}{d_v+d_{v_j}}}-\sqrt{1-\frac{2}{d_u+d_v+d_{v_j}-2}\,}\, \right)\nonumber\\[2mm]
 &\quad\ +\sqrt{1-\frac{2}{d_u+d_{v}}}\, .
\end{align}
By Lemma \ref{p1}, the following inequalities hold for all $i=1,...,d_u-1$ and $j=1,...,d_v-1$:
$$\sqrt{1-\frac{2}{d_u+d_{u_i}}}-\sqrt{1-\frac{2}{d_u+d_v+d_{u_i}-2}\,}\geq
\sqrt{1-\frac{2}{d_u+1}}-\sqrt{1-\frac{2}{d_u+d_v-1}\,}$$
and
$$\sqrt{1-\frac{2}{d_v+d_{v_j}}}-\sqrt{1-\frac{2}{d_u+d_v+d_{v_j}-2}\,}\geq
\sqrt{1-\frac{2}{d_v+1}}-\sqrt{1-\frac{2}{d_u+d_v-1}\,}\,.$$
Thus, Equation \eqref{Eq-1} yields
\begin{align}\label{Eq-2}
 ABS(T)-ABS(T')&\ge(d_u-1)\left( \sqrt{1-\frac{2}{d_u+1}}-\sqrt{1-\frac{2}{d_u+d_v-1}\,}\, \right)\nonumber\\[2mm]
 &\quad\ +(d_v-1)\left( \sqrt{1-\frac{2}{d_v+1}}-\sqrt{1-\frac{2}{d_u+d_v-1}\,}\, \right)\nonumber\\[2mm]
 &\quad\ +\sqrt{1-\frac{2}{d_u+d_{v}}}\, .
\end{align}
Since $d_u\ge d_v$, by Lemma \ref{lem-2}, Inequality \eqref{Eq-2} gives
\begin{align}\label{Eq-3}
  ABS(T)-ABS(T')&\ge2(d_u-1)\left( \sqrt{1-\frac{2}{d_u+1}}-\sqrt{1-\frac{2}{2d_u-1}\,}\, \right)\nonumber\\[3mm]
  &\quad \ +\sqrt{1-\frac{1}{d_u}}\,.
\end{align}
By Lemma \ref{lem-2} the function $g(s)$ defined by
\[
g(s) = f(s,s)=2(s-1)\left( \sqrt{1-\frac{2}{s+1}}-\sqrt{1-\frac{2}{2s-1}\,}\, \right) +\sqrt{1-\frac{1}{s}}\,,
\]
with $s\ge3$, is strictly decreasing, and
\[
\lim_{s\rightarrow \infty }g(s)=\lim_{s\rightarrow \infty} \left[\frac{-4s^2+12s-8}{\sqrt{2s^2+s-1}\left(\sqrt{2s^2-3s+1}+\sqrt{2s^2-s-3}\right)}+\sqrt{1-\frac{1}{s}}\right]=0.
\]
Therefore, the right-hand side of \eqref{Eq-3} is positive for $d_u\ge3$. This completes the proof.

\end{proof}

\section{Trees With a Fixed Order and Pendent Vertices}\label{sec2}

In this section, we characterize  trees attaining the minimum value of the ABS index in $\Gamma_{n,p}$, where $p\geq3$ and $n\geq3p-2$.

\begin{lemma}\label{l1-new}
If $y$ is a fixed real number greater than or equal to $3$ then the function $f$, defined in Lemma \ref{p1}, is strictly increasing in $x$.
\end{lemma}

For a tree $T$, denote by $W_1(T)$ the set of pendent vertices of $T$, $W_2(T)=\cup_{v\in W_1(T)}N(v)$, and $W_3(T)=V(T)\setminus (W_1(T)\cup W_2(T))$. A vertex in $T$ of degree at least three is called a branching vertex. A path is called an internal path, if its end vertices are branching vertices and every other vertex has degree two. A path is called a pendent path if one of the end vertices is pendent and every other vertex has degree two.

\begin{lemma}\label{p2}
Let $T^{\ast}\in \Gamma_{n,p}$ attaining minimum $ABS$ value. Then every internal path of $T^{\ast}$ has length one.
\end{lemma}
\begin{proof}
For a contradiction, assume $T^{\ast}$ contains an internal path $v=v_0-v_1-v_2-..-v_r=u$ of length $r\geq 2$. Let $w\in W_2(T)$ and let $y$ be a pendent vertex adjacent to $w$.  Let $T'=T^{\ast}-\{vv_1,uv_{r-1}\}+\{vu,yv_1\}$. Clearly  $T'\in \Gamma_{n,p}$. In
what follows, by $d_x$ we denote the degree of a vertex $x$ in $T^{\ast}$.  If $r>2$, then
\begin{align}\label{eq-01A}
ABS(T')-ABS(T^{\ast})&= f(d_{v},d_u)+f(2,2)+f(d_w,2)-f(d_v,2)\nonumber\\
&-f(d_u,2)-f(2,2)+f(1,2)-f(d_w,1)\nonumber\\
&=(f(d_v,d_u)-f(d_v,2))-(f(2,d_u)-f(2,2))\nonumber\\&
+(f(d_w,2)-f(d_w,1))-(f(2,2)-f(2,1))\nonumber\\
&=\psi(2,d_v;d_u-2)-\psi(2,2;d_u-2)+\psi(1,d_w;1)-\psi(1,2;1)<0,\nonumber
\end{align}
which yields a contradiction. If $r=2$ then
\begin{align}
ABS(T')-ABS(T^{\ast})&= f(d_{v},d_u)+f(2,1)+f(d_w,2)-f(d_v,2)-f(d_u,2)-f(d_w,1)\nonumber\\
&=(f(d_v,d_u)-f(d_v,2))-(f(2,d_u)-f(2,2))\nonumber\\&
+(f(d_w,2)-f(d_w,1))-(f(2,2)-f(2,1))\nonumber\\
&=\psi(2,d_v;d_u-2)-\psi(2,2;d_u-2)+\psi(1,d_w;1)-\psi(1,2;1)<0.\nonumber
\end{align}
Again this contradicts the minimality of $ABS(T^{\ast})$. Thus $T^{\ast}$ does not have an internal path of length greater than one.
\end{proof}

\begin{lemma}\label{p3}
Let $T^{\ast}\in \Gamma_{n,p}$ attaining  minimum $ABS$ value. Then $d_u\leq d_v$ for every $u\in W_2(T^{\ast})$ and $v\in W_3(T^{\ast})$.
\end{lemma}
\begin{proof}
Let $y$ be a pendent vertex adjacent to $u$ and let $x$ be a non-pendent vertex adjacent to $v$ and not on the $u-v$ path. Let $T'=T^{\ast}-\{xz\mid z\in N(x)\setminus\{v\}\}+\{zy\mid z\in N(x)\setminus\{v\}\}$. In
what follows, by $d_x$ we denote the degree of a vertex $x$ in $T^{\ast}$. Clearly, $T'\in \Gamma_{n,p}$, and so $ABS(T')\geq ABS(T^{\ast})$. Thus we have,
\begin{equation*}
0\leq ABS(T')-ABS(T^{\ast})= f(d_{v},1)-f(d_{v},d_x)+f(d_u,d_x)-d(d_u,1)=\psi(1,d_u;d_x-1)-\psi1,d_v,d_x-1).
\end{equation*}
Hence $d_u\leq d_v$.
\end{proof}

\begin{lemma}\label{p4}
Let $T\in \Gamma_{n,p}$. If $n=3p-2+t$ for some $t\geq0$, then the following assertions hold.
\begin{enumerate}[label=(\roman*)]
\item\label{p4i} $d_v=2$ for all $v\in W_2(T)$
\item\label{p4ii} $\sum_{u\in W_3(T)}d_u=3p-6+2t.$
\end{enumerate}
\end{lemma}
\begin{proof}
\ref{p4i} We have $p+\sum_{v\in W_2(T)}d_v+\sum_{u\in W_3(T)}d_u=6p-6+2t$. Seeking a contradiction, assume $d_{v_0}\geq3$ for some ${v_0}\in W_2(T)$. Let $\beta=min\{d_u\mid u\in W_3(T)\}$. Then by Lemma \ref{p3}, $\beta \geq d_{v_0}\geq 3$. On the other hand, since $d_v\geq 2$ for all $v\in W_2(T)$, we have $$p+2(|W_2(T)|-1) +3+\beta(2p-2+t-|W_2(T)|)\leq 6p-6+2t.$$ So $\beta\leq \frac{5p-7+2t-2|W_2(T)|}{2p-2+t-|W_2(T)|}=\frac{p-3}{2p-2+t-|W_2(T)|}+2<3$, a contradiction. Thus  $d_v=2$ for all $v\in W_2(T)$.

\ref{p4ii} We conclude from part \ref{p4i} that $|W_2(T)|=p$. So $p+2p+\sum_{u\in W_3(T)}d_u=6p-6+2t$, which yields $\sum_{u\in W_3(T)}d_u=3p-6+2t.$
\end{proof}

\begin{remark}\label{r1} Let $T^{\ast}\in \Gamma_{n,p}$ be a tree attaining minimum $ABS$ value. It results from Lemma \ref{p3} that every vertex of $T^{\ast}$ of degree two is on a pendent path. Assume there are two vertices $u,v\in W_3(T^{\ast})$ such that $d_u=d_v=2$ and $u$ and $v$ do not belong to the same pendent path. Since $v\in W_3(T^{\ast})$ and $d_v=2$, there are two vertices $x$ and $y$ such that $d_x\leq d_y=2$ and $v-y-x$. Let $z$ be the pendent vertex at the end of the pendent path containing $u$ and  let $T'=T^{\ast}-{xy}+{zy}$. Clearly $T'\in\Gamma_{n,p}$ and $ABS(T')=ABS(T^{\ast})$. We conclude that it is possible to obtain a tree $T_1^{\ast}\in \Gamma_{n,p}$, in which all vertices of degree two in $W_3(T_1^{\ast})$ belong to the same pendent path, such that $ABS(T_1^{\ast})=ABS(T^{\ast})$. Moreover, $T^{\ast}$ and $T_1^{\ast}$ have the same maximum degree and the subtrees of $T^{\ast}$ and $T_1^{\ast}$ induced on their  sets of branching vertices are isomorphic.
\end{remark}

\begin{theorem}\label{txd}
Let $p\geq3$ and $n\geq 3p-2$. If $T^{\ast}\in \Gamma_{n,p}$ attaining the minimum $ABS$ value, then the maximum degree of $T^{\ast}$ is $3$.
\end{theorem}
\begin{proof}
For a contradiction, assume $T^{\ast}$ has a vertex of degree at least $4$. Let $\beta=min\{d_v\mid v\in W_3(T)\}$. Then, by Lemma \ref{p4} \ref{p4ii}, $4+\beta(|W_3(T^{\ast})|-1)\leq 3p-6+2t$. Consequently, $\beta\leq \frac{3p-10+2t}{p-3+t}<3$. So there is a vertex $u\in W_3(T^{\ast})$ such that $d_u=2$. We select $u$ so that one of its neighbors has degree at least three. Now let $v$ be the farthest vertex from $u$ that has degree at least $4$. Clearly, for each vertex $x\in N(v)$ that does not lie on the $u-v$ path, we have $d_x\leq 3$. In what follows, $d_a$ denotes the degree of a vertex $a$ in $T^{\ast}$.\\
Case 1. $u\in N(v)$. In this case choose $y\in N(v)\setminus\{u\}$ and let $T'=T^{\ast}-\{yv\}+\{yu\}$. Then
\begin{align} ABS(T')-ABS(T^{\ast})= &f(d_y,3)-f(d_y,d_v)+f(2,3)-f(2,2)\nonumber\\
&+f(3,d_v-1)-f(2,d_v)\nonumber\\
&+\sum_{x\in N(v)\setminus\{y,u\}}(f(d_x,d_v-1)-f(d_x,d_v))\nonumber\\
\leq &f(3,3)+f(2,3)-f(2,2)-f(3,d_v)\nonumber\\
&-(d_v-2)(f(3,d_v)-f(3,d_v-1)).\nonumber
\end{align}
Consider the function $$h_1(s)=f(3,3)+f(2,3)-f(3,s)-f(2,2)-(s-2)(f(3,s)-f(3,s-1)).$$ We have
$$h_1'(s)=\frac{A_1-B_1}{\sqrt{s(s + 1)}(s+2)^2(s+3)^2},$$

where $A_1=(s^2+3s-2)(s+3)^2\sqrt{(s+1)(s+2)}$ and $B_1=(s^2+5s+2)(s+2)^2\sqrt{s(s+3)}$. Since
$A^2_1-B_1^2=-14s^5-137s^4-456s^3-577s^2-140s+108\leq 0$,
for $s\geq 1$, we get $h_1'(s)\leq 0$ for $s\geq 1$, and thus $h_1(s)$ is strictly decreasing for $s\geq1$. This implies that $ABS(T')-ABS(T^{\ast})< h_1(d_v)\leq h_1(4)<0$, a contradiction.

Case 2. $u\not\in N(v)$. Let $w\in N(v)$ and $z\in N(u)$ such that $w$ and $z$ lie on the $u-v$ path. By Case 1, me may suppose that $d_z\geq 3$. Now we divide this case to four subcases.

Subcase 2.1. $d_v \geq 5$.  Let $y\in N(v)\setminus \{w\}$ and take $T'=T^{\ast}-\{yv\}+\{yu\}$. Then $T'\in \Gamma_{n,p}$. Moreover, from Lemma \ref{p1}, it follows that
\begin{align} ABS(T')-ABS(T^{\ast})= &f(d_y,3)-f(d_y,d_v)+f(2,3)-f(2,2)\nonumber\\
&+f(d_z,3)-f(d_z,2)\nonumber\\
&+\sum_{x\in N(v)\setminus\{y\}}(f(d_x,d_v-1)
-f(d_x,d_v))\nonumber\\
\leq &2f(3,3)-f(3,d_v)-f(2,2)-(d_v-2)(f(3,d_v)\nonumber\\
&-f(3,d_v-1))+f(d_w,d_v-1)-f(d_w,d_v).\nonumber
\end{align}
The function $h_2(s)=2f(3,3)-f(3,s)-f(2,2)-(s-2)(f(3,s)-f(3,s-1))$ is strictly decreasing for $s\geq 1$, and so $ABS(T')-ABS(T^{\ast})< h_2(5)<0,$ a contradiction.


Subcase 2.2. $d_v=4$ and $d_w\leq 3$. Let $y\in N(v)\setminus \{w\}$ and take $T'=T^{\ast}-\{yv\}+\{yu\}$. Then $T'\in \Gamma_{n,p}$. Moreover,
\begin{align}
ABS(T')-ABS(T^{\ast})&=f(d_y,3)-f(d_y,4)+f(2,3)-f(2,2)\nonumber\\
&+f(d_z,3)-f(d_z,2)+\sum_{x\in N(v)\setminus\{y\}}(f(d_x,3)-f(d_x,4))\nonumber\\
&\leq 5f(3,3)-4f(3,4)-f(2,2)<0,
\end{align}
 a contradiction.

Subcase 2.3: $d_v=4$ and $d_w\geq 6$. Note that $d_x\leq 4$ for each $x\in N_w\setminus \{w_1,v\}$,  because otherwise we reach a contradiction as in Subcase 2.1, where $w_1$ is the unique neighbor of $w$ in the $u-v$ path. Let $T'=T^{\ast}-\{vw\}+\{vu\}$. Then
\begin{align} ABS(T')-ABS(T^{\ast})= &f(4,3)-f(4,d_w)
+f(2,3)-f(2,2)+f(d_z,3)-f(d_z,2)\nonumber\\
&+\sum_{x\in N(w)\setminus\{v\}}(f(d_x,d_w-1)-f(d_x,d_w))\nonumber\\
<&f(4,3)-f(4,d_w)+f(3,3)-f(2,2)\nonumber\\
&-(d_w-1)(f(4,d_w)-f(4,d_w-1)).
\end{align}

The function $$h_4(s)=f(4,3)-f(4,s)+f(3,3)-f(2,2)-(s-1)(f(4,s)-f(4,s-1))$$ is strictly decreasing for $s\geq1$. Thus $ABS(T')-ABS(T^{\ast})< h_4(6)<0$, a contradiction.

Subcase 2.4: $d_v=4$ and $4\leq d_w\leq 5$. Let $w_1$ be the unique neighbor of $w$ in the $u-v$ path, and let $b_1,...,b_k$ be all vertices in $N(w)\setminus\{w_1\}$ of degree $4$ (with $b_1=v$). Then by Lemma \ref{p4} \ref{p4ii}, there are $k+1$ distinct vertices $x_i\in W_3(T^{\ast})$ such that $d_{x_i}=2$. The vertex $x_{k+1}$ exists because $d_w\geq4$. By Remark \ref{r1}, we may assume that these vertices form a path $u=x_1-x_2-...-x_k-x_{k+1}$.
 For each $i=2,...,k$ select $c_i\in N(b_i)\setminus\{w\}$ and let $T'=T^{\ast}-(\{c_ib_i,i=2,...,k\}\cup\{vw\})+(\{c_ix_i,i=2,...,k\}\cup\{vu\})$. Note that since $v$ is closer to $u$ than $c_i$, we have $d_{c_i}\leq 3$. Thus
\begin{align} ABS(T')-ABS(T^{\ast})=&
f(2,3)-f(2,2)+f(d_z,3)-f(d_z,2)\nonumber\\
&+f(4,3)-f(4,d_w)+(k-1)(f(3,3)-f(2,2))\nonumber\\
&+(k-1)(f(4,d_w)-f(3,d_w-1))\nonumber\\
&+\sum_{i=2}^{k}(f(d_{c_i},3)-f(d_{c_i},4))\nonumber\\
&+\sum_{i=2}^{k}\sum_{x\in N(b_i)\setminus\{w,c_i\}}(f(d_x,3)-
f(d_x,4))\nonumber\\
&+\sum_{e\in N(w)\setminus\{b_2,...,b_k,v\}}(f(d_e,d_w-1))-
f(d_e,d_w))\nonumber\\
\leq &(4k-2)f(3,3)-3(k-1)f(3,4))-kf(2,2)-f(4,d_w)\nonumber\\
&-(k-1)(f(4,d_w)-f(3,d_w-1))\nonumber\\
&-(d_w-k)(f(3,d_w)-f(3,d_w-1)).\nonumber
\end{align}
Calculations show that the right hand side of this inequality is negative when $d_w\in\{4,5\}$ and $k\in\{1,...,d_w\}$. This also yields a contradiction. Therefore the maximum degree of $T^{\ast}$ is $3$ as desired.
\end{proof}

Now we are ready to characterize the trees that minimize $ABS$ index in $\Gamma_{n,p}$ where $p\geq3$ and $n\geq3p-2$. For $p\geq3$ and $n=3p-2+t$ with $t\ge0$, let $\Gamma_{n,p}^* \subset \Gamma_{n,p}$ such that an arbitrary tree $T^{\ast}$ belonging to the class $\Gamma_{n,p}^*$ is defined as follows.\\
(1) Let $T_0$ be a tree of order $p-2$ and maximum degree $\Delta=min\{p-3,3\}$.\\
(2) Construct a tree $T_1$ from $T_0$ by attaching $3-i$ pendent path(s) of length two at each vertex of degree $i$ for $i=1,2,3$. (Note that the order of $T_1$ is $5n_1(T_0)+3n_2(T_0)+n_3(T_0)$, which is equal to $3p-2$ because $n_1(T_0)+n_2(T_0)+n_3(T_0)=p-2$ and $n_1(T_0)+2n_2(T_0)+3n_3(T_0)=2(p-3)$, where $n_k(T_0)$ denotes the number of vertices of degree $k$ in $T_0$.)\\
(3) Finally $T^{\ast}$ is obtained from $T_1$ by inserting $t$ vertices of degree two into one or more pendent paths.\\

Figure \ref{Fig-3} gives a tree of the class $\Gamma_{n,p}^*$.

\begin{figure}[!ht]
\centering
  \includegraphics[width=0.6\textwidth]{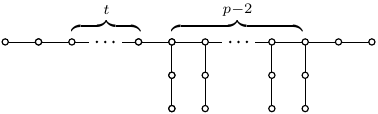}
   \caption{A tree attaining the minimum value of the ABS index in the class $\Gamma_{n,p}$ where $p\geq3$ and $n\ge3p-2$.}
    \label{Fig-3}
     \end{figure}

\begin{theorem}
Let $p\geq3$ and $n=3p-2+t$ where $t\geq0$. If $T\in \Gamma_{n,p}$ then $$ABS(T)\geq \left(\frac{2}{\sqrt{6}}+\sqrt{\frac{3}{5}}+\frac{1}{\sqrt{3}}\right)p+\frac{\sqrt{2}}{2}t+\frac{4}{\sqrt{6}},$$
with equality if and only if $T\in \Gamma_{n,p}^*$.
\end{theorem}

\begin{proof}
Let $T^{\ast}\in \Gamma_{n,p}$ having the minimum $ABS$ value. Then by Theorem \ref{txd}, the maximum degree in $T$ is three. For $i,j=1,2,3$ let $n_{i,j}$ be the number of edges in $T^{\ast}$ that joints vertices of degrees $i$ and $j$. Clearly, $n_{1,1}=0$. From Lemma \ref{p4} \ref{p4i}, we get $n_{1,2}=p$ and $n_{1,3}=0$. Let $d$ be the number of vertices of degree $3$. Then by Lemma \ref{p4} \ref{p4ii}, we get $2(p-2+t-d)+3d=3p-6+2t$. Thus $d=p-2$.  Since $T$ has no internal paths of length more than one, then the induced subgraph of $T$ over the set of vertices of degree three is a tree on $p-2$ vertices. Thus $n_{3,3}=p-3$. There are $3(p-2)$ edges incident with the vertices of degree $3$. Each one of these edges is either joining two vertices of degree $3$ or a vertex of degree $3$ and a vertex of degree $2$. Thus, we get $n_{2,3}+2n_{3,3}=3(p-2)$, and hence  $n_{2,2}=3(p-2)-2(p-3)=p$. By subtracting the values $n_{1,2}=p$, $n_{2,3}=p$, $n_{3,3}=p-3$ from $n-1=3p-3+t$, we get $n_{2,2}=t$. Thus
\begin{align}
ABS(T^{\ast})&=pf(1,2)+tf(2,2)+pf(2,3)+(p-2)f(3,3)\nonumber\\
             &=\left(\frac{2}{\sqrt{6}}+\sqrt{\frac{3}{5}}+\frac{1}{\sqrt{3}}\right)p+\frac{\sqrt{2}}{2}t+\frac{4}{\sqrt{6}}.\nonumber
\end{align}
\end{proof}

Here we remark that the statements of Lemma 3.11 and Corollary 3.12 in paper \cite{X} concerning chemical trees are not complete; these statements should include the condition $n\geq 3p-2$. Indeed, Fig. \ref{Graphs}  gives the example of (general and chemical) trees with maximum degree $\Delta\geq4$ attaining minimum ABS value in the class $\Gamma_{n,p}$ with $n\leq 3p-3$. Thus, the condition $n\geq 3p-2$ in Theorem \ref{txd} and in the aforementioned two results of \cite{X} is necessary. The problem of characterizing trees attaining the minimum values of the ABS index in $\Gamma_{n,p}$ under the conditions $p\geq 3$ and $p+1\leq n\leq 3p-3$ is still open (the maximal version of this problem was recently solved in \cite{Noureen-23}). In the remainder of this section, we address this problem. Note that $\Gamma_{p+1,p}$ consists of only one tree, namely the star tree. The next theorem deals with the case $n=p+2$.

\begin{theorem}
Let $p\geq 3$. Then the minimum $ABS$ index in $\Gamma_{p,p+2}$ is attained uniquely by the balanced double star.
\end{theorem}

\begin{proof} Let $T^{\ast}\in \Gamma_{p,p+2}$ attaining minimum $ABS$ index. Assume that $u,v\in V(T^{\ast})$ are the non-pendent vertices with $d_u\geq d_v+2$. Let $w\in N(u)\setminus\{v\}$ and let $T'=T-\{uw\}+\{vw\}$. Then

\begin{align*} ABS(T')-ABS(T^{\ast})=&
(f(1,d_v+1)-f(1,d_v))d_v-(d_u-1)(f(1,d_u)-f(1,d_u-1))\nonumber\\
&+f(1,d_v+1)-f(1,d_u)\nonumber\\
&=g(d_v)-g(d_u-1);
\end{align*}
where $g(x)=x(f(1,x+1)-f(1,x))+f(1,x+1)$. The derivative of $g(x)$ is given by
$$g'(x)=\frac{A-B}{(x+1)^2(x+2)^2\sqrt{x(x-1)}},$$
where $A=(x^2+3x+1)(x + 1)^2\sqrt{(x - 1)(x + 2)}$
and $B=(x^2+x-1)(x+2)^2\sqrt{x(x + 1)}$.
Since $A^2-B^2>0$ for $x\geq 1$ we get that $g(x)$ is decreasing. So $ABS(T')-ABS(T^{\ast})<0$, a contradiction.
\end{proof}

We end this section by giving all the graphs attaining minimum ABS index in the class $\Gamma_{n,p}$ for $p=4,5,6$ and $p+3\leq n \leq 3p-3$, see Figure \ref{Graphs}. These graphs are obtained by utilizing a computer software.
Observe that for every $(n,p)\not\in\{(8,4), (9,5),(13,6)\}$ with $p=4,5,6$ and $p+3\leq n \leq 3p-3$,
there is the unique graph attaining minimum ABS index in the class $\Gamma_{n,p}$; however, for every $(n,p)\in\{(8,4), (9,5),(13,6)\}$, there exist exactly two such extremal graphs.
From the trees depicted in Figure \ref{Graphs}, one may expect some certain structural properties of a tree attaining minimum ABS index in the class $\Gamma_{n,p}$ when $p\ge7$ and $p+3\leq n \leq 3p-3$. But, these trees seem to be insufficient for making some sound conjectures.

\begin{figure}[htp]
 \centering
  \includegraphics[width=0.9\textwidth]{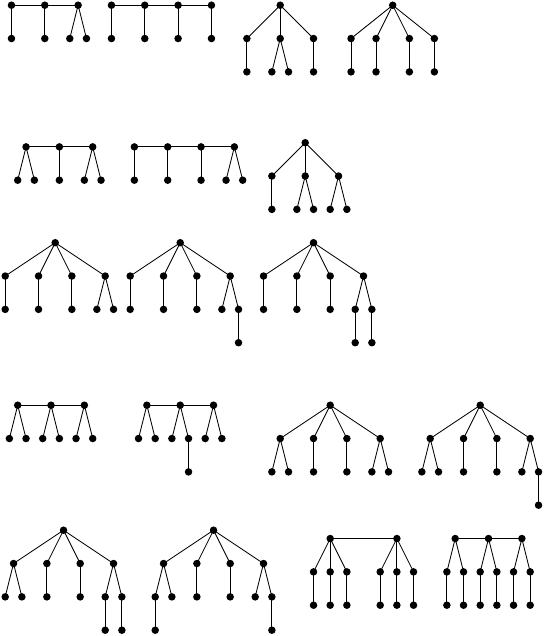}\\
   \caption{Graphs attaining minimum ABS value in the class $\Gamma_{n,p}$ for $p=4,5,6$ and $p+3\leq n \leq 3p-3$.}
    \label{Graphs}
     \end{figure}
\section*{Use of AI tools declaration}
The authors declare they have not used Artificial Intelligence (AI) tools in the creation of this article.

\section*{Acknowledgments}
This work is supported by Scientific Research Deanship, University of Ha\!'il, Saudi Arabia, through project number RG-23\,093.

\section*{Conflict of interest}
The authors declare that they do not have any conflict of interest.


\begin{thebibliography}{999}
\bibitem{Gross-05} J. L. Gross, J. Yellen, {\em Graph Theory and Its Applications}, Second Edition, CRC Press, 2005.

\bibitem{Wagner-18} S. Wagner, H. Wang, {\em Introduction to Chemical Graph Theory}, CRC Press, Boca Raton, 2018.

\bibitem{Randic-75} M. Randi\'c, {On characterization of molecular branching},
         \newblock  {\it J. Am. Chem. Soc.}, \newblock  \textbf{97} (1975), 6609--6615. \url{https://doi.org/10.1021/ja00856a001}

\bibitem{Gutman-13} I. Gutman, {Degree-based topological indices}, \newblock  {\it Croat. Chem. Acta.}, \newblock  \textbf{86} (2013), 351--361.
\url{http://dx.doi.org/10.5562/cca2294}

\bibitem{Li-MATCH-08} X. Li, Y. Shi, {A survey on the Randi\'c index},
        \newblock  {\it MATCH Commun. Math. Comput. Chem.}, \newblock  \textbf{59} (2008), 127--156. \url{https://match.pmf.kg.ac.rs/electronic_versions/Match59/n1/match59n1_127-156.pdf}

\bibitem{Randic-01} M. Randi\'c, {The connectivity index 25 years after},
        \newblock  {\it J. Mol. Graph. Model.}, \newblock  \textbf{20} (2001), 19--35. \url{https://doi.org/10.1016/S1093-3263(01)00098-5}

\bibitem{new1} I. Gutman, B. Furtula (Eds.), {\em Recent Results in the Theory of Randi\'c Index}, Math. Chem. Monogr. 6, Univ. Kragujevac, Kragujevac, 2008.

\bibitem{new3} X. Li, I. Gutman, {\em Mathematical Aspects of Randi\'c-Type Molecular Structure Descriptors}, Univ. Kragujevac, Kragujevac, 2006.

\bibitem{Estrada-98} E. Estrada, L. Torres, L. Rodr\'{\i}guez, I. Gutman,
        {An atom-bond connectivity index: modelling the enthalpy of formation of alkanes},  \newblock  {\it Indian J. Chem. Sec. A}, \newblock  \textbf{37} (1998), 849--855.

\bibitem{Estrada-08} E. Estrada, {Atom-bond connectivity and the energetic of
        branched alkanes}, \newblock  {\it Chem. Phys. Lett.}, \newblock  \textbf{463} (2008), 422--425. \url{https://doi.org/10.1016/j.cplett.2008.08.074}

\bibitem{Zhou-JMC-09} B. Zhou, N. Trinajsti\'c, {On a novel connectivity index},
        \newblock  {\it J. Math. Chem.}, \newblock  \textbf{46} (2009), 1252--1270. \url{https://doi.org/10.1007/s10910-008-9515-z}

\bibitem{Ali-DML-21} A. Ali, K. C. Das, D. Dimitrov, {B. Furtula, Atom-bond connectivity index of graphs: a review over extremal results and bounds},
        \newblock  {\it Discrete Math. Lett.}, \newblock  \textbf{5} (2021), 68--93. \url{http://dx.doi.org/10.47443/dml.2020.0069}

\bibitem{Ali-MATCH-19} A. Ali, L. Zhong, I. Gutman, {Harmonic index and its
        generalization: extremal results and bounds}, \newblock  {\it MATCH Commun. Math. Comput. Chem.}, \newblock  \textbf{81} (2019), 249--311. \url{https://match.pmf.kg.ac.rs/electronic_versions/Match81/n2/match81n2_249-311.pdf}

\bibitem{ABS-JMC} A. Ali, B. Furtula, I. Red\v{z}epovi\'c, I. Gutman,
        {Atom-bond sum-connectivity index}, \newblock  {\it J. Math. Chem.}, \newblock  \textbf{60} (2022), 2081--2093. \url{https://doi.org/10.1007/s10910-022-01403-1}

\bibitem{Tang-DAM-16} Y. Tang, D. B. West, B. Zhou, {Extremal problems for
        degree-based topological indices},  \newblock  {\it Discrete Appl. Math.}, \newblock  \textbf{203} (2016), 134--143. \url{https://doi.org/10.1016/j.dam.2015.09.011}

\bibitem{ABS-EJM} A. Ali, I. Gutman, I. Red\v{z}epovi\'c, {Atom-bond sum-connectivity index of unicyclic graphs and some applications}, \newblock  {\it Electron. J. Math.}, \newblock  \textbf{5} (2023), 1--7. \url{https://doi.org/10.47443/ejm.2022.039}

 \bibitem{Bondy08} J. A. Bondy, U. S. R. Murty, {\em Graph Theory}, Springer, 2008.


\bibitem{X} J. Du, X. Sun, {On bond incident degree index of chemical trees with a fixed order and a fixed number of leaves}, \newblock  {\it Appl. Math. Comp.}, \newblock  \textbf{464} (2024), 128390. \url{https://doi.org/10.1016/j.amc.2023.128390}

\bibitem{Noureen-23} S. Noureen, A. Ali, Maximum atom-bond sum-connectivity index of
n-order trees with fixed number of leaves, Discrete Math. Lett. {\bf12} (2023), 26--28.
\url{https://doi.org/10.47443/dml.2023.016}



\end{thebibliography}
\end{document}